\documentclass[11pt]{article}
\usepackage[utf8]{inputenc}
\usepackage[T1]{fontenc}
\usepackage{graphicx}
\usepackage{longtable}
\usepackage{wrapfig}
\usepackage{rotating}
\usepackage[normalem]{ulem}
\usepackage{amsmath}
\usepackage{amssymb}
\usepackage{capt-of}
\usepackage{hyperref}
\usepackage{subfigure}
\usepackage{amsthm}
\newtheorem{theorem}{Theorem}[section]
\newtheorem{lemma}[theorem]{Lemma}

\usepackage{listings}
\author{Rafael Villarroel-Flores\\ Universidad Autónoma del Estado de Hidalgo\\ Carretera Pachuca-Tulancingo km. 4.5\\ Pachuca 42184 Hgo. México\\ email: \texttt{rafaelv@uaeh.edu.mx}keywords: helly property, clique graphs}
\date{\today}
\title{On the clique behavior and Hellyness of the complements of regular graphs}
\hypersetup{
 pdfauthor={Rafael},
 pdftitle={On the clique behavior and Hellyness of the complements of regular graphs},
 pdfkeywords={},
 pdfsubject={},
 pdfcreator={Emacs 28.1 (Org mode 9.5.2)}, 
 pdflang={English}}
\begin{document}

\maketitle
\begin{abstract}
A collection of sets is intersecting, if any pair of sets in the collection has nonempty intersection.
A collection of sets \(\mathcal{C}\) has the Helly property if any intersecting subcollection has nonempty intersection. A graph is \emph{Helly} if the collection of maximal complete subgraphs of \(G\) has the Helly property. We prove that if \(G\) is a \(k\)-regular graph with \(n\) vertices such that \(n>3k+\sqrt{2k^{2}-k}\), then the complement \(\overline{G}\) is not Helly. We also consider the problem of whether the properties of Hellyness and convergence under the clique graph operator are equivalent for the complement of \(k\)-regular graphs, for small values of \(k\).
\end{abstract}

\section{Introduction}
\label{sec:orgbd8452c}

In this paper we consider only finite simple graphs. As a reference for Graph Theory, we follow \cite{MR0256911-djvu}. We usually identify subsets of vertices of a graph \(G\) with the subgraph they induce. A \emph{complete} in a graph \(G\) is a set of vertices which are adjacent by pairs, and a \emph{clique} is a maximal complete. The \emph{clique graph} \(K(G)\) of a graph \(G\) is the intersection graph of the cliques of \(G\). We can then define a sequence of \emph{iterated clique graphs} as: \(K^{0}(G)=G\), \(K^{n}(G)=K(K^{n-1}(G))\) for \(n\geq1\). If this sequence of graphs has a finite number of different graphs up to isomorphism, we say that \(G\) is \emph{convergent}, otherwise, \(G\) is said to be \emph{divergent}. There is a number of criteria in the literature in order to determine which of these \emph{behaviors} correspond to a given graph \(G\), however it is conjectured that the problem in general is algorithmically unsolvable for finite graphs (\cite{2021-cedillo-pizana-clique-convergence-is-undecidable-for-automatic-graphs}).

A collection of sets is \emph{intersecting} if any two members of the collection has nonempty intersection. The collection has the \emph{Helly property} if any intersecting subcollection has nonempty intersection. We say that a graph \(G\) is \emph{Helly} if the collection of the cliques of \(G\) has the Helly property (in the literature, this is usually called the \emph{clique-Helly property}). The Helly property on cliques can be tested in polynomial time (\cite{1989-dragan-centers-of-graphs-and-the-helly-property},  \cite{1997-szwarcfiter-recognizing-clique-helly-graphs}), and has been studied in several papers (for example, \cite{2009-dourado-protti-szwarcfiter-complexity-aspects-of-the-helly-property}, \cite{2007-lin-szwarcfiter-faster-recognition-of-clique-helly-and-hereditary-clique-helly-graphs}) and generalized (\cite{2007-dourado-protti-szwarcfiter-characterization-and-recognition-of-generalized-clique-helly-graphs}, \cite{dourado-grippo-safe-2022-on-the-generalized-helly-property-of-hypergraphs-cliques-and-bicliques}).

A Helly graph is convergent (\cite{MR0329947}) but there are convergent graphs which are not Helly. However, there are several graph classes for which it is has been proven that a graph in such a class is convergent if and only if it is Helly. Such classes include that of cographs (\cite{2004-larrion-mello-morgana-neumann-lara-pizana-the-clique-operator-on-cographs-and-serial-graphs}), complements and powers of cycles (\cite{zbMATH03641500}, \cite{MR2489268}), chessboard graphs (\cite{chessboard-graphs}) and circulants with three small jumps (\cite{2014-larrion-pizana-villarroel-flores-the-clique-behavior-of-circulants-with-three-small-jumps}).

In this note we consider a class of regular graphs. We want to study regular graphs with high degree, hence it will be convenient for us to consider the complement. We prove that for \(k=1,2\), the complements of \(k\)-regular graphs are convergent if and only if they are Helly, and we show that this happens exactly when the order of the graph is sufficiently small. Then for each \(k\geq 3\), we determine a number \(N(k)\) such that all \(k\)-regular graphs with at least \(N(k)\) vertices are such that their complements are not Helly. In contrast to the cases \(k=1,2\), we can show examples of \(3\)-regular graphs \(G\) such that \(\overline{G}\) is not Helly, and still \(\overline{G}\) is convergent. In fact, we conjecture that there are arbitrarily large \(3\)-regular graphs \(G\) where the complement \(\overline{G}\) is convergent but not Helly.

\section{Preliminaries}
\label{sec:orgbc1ced3}

\subsection{Definitions}
\label{sec:orgf2aeb8c}

The \emph{disjoint union} \(G\cup H\) of two graphs \(G,H\) is the graph that has as vertex set the disjoint union of the vertex sets of \(G\) and \(H\), and as edge set the union of the edge sets of \(G\) and \(H\). We can extend this definition to the disjoint union of a finite collection of graphs. The disjoint union of \(m\) copies of the graph \(G\) is denoted by \(mG\). We denote by \(G+H\) the graph obtained from the disjoint union of \(G\) and \(H\), adding all edges between a vertex in \(G\) and a vertex in \(H\). It is immediate that \(\overline{G\cup H}=\overline{G}+\overline{H}\).

The \emph{\(m\)-th octahedron} \(O_{m}\) is the complement of the graph \(mK_{2}\).

A \emph{triangle} in a graph \(G\) is a complete with three vertices. If \(T\) is a triangle, then its \emph{extended triangle} \(\hat{T}\) in \(G\) is the induced subgraph of \(G\) on the vertices that are adjacent to at least two elements of \(T\). A graph is a \emph{cone} if it has a vertex adjacent to all other vertices in the graph.

A \emph{cotriangle} in a graph \(G\) is an independent set of three of its vertices. 

If \(G\) is a graph, and \(\sigma\colon G\to G\) is an automorphism, we say that \(\sigma\) is a \emph{coaffination} if \(\sigma(x)\ne x\) and \(\sigma(x)\) is not a neighbor of \(x\) for every \(x\in G\). Note that the complement of every cyclic graph has a coaffination.

\subsection{Theorems on clique behavior}
\label{sec:org9f2d4e0}

The following theorem appeared in \cite{1989-dragan-centers-of-graphs-and-the-helly-property} but it was proved independently in \cite{1997-szwarcfiter-recognizing-clique-helly-graphs}.

\begin{theorem}
\label{helly-extended}
(\cite{1989-dragan-centers-of-graphs-and-the-helly-property}, \cite{1997-szwarcfiter-recognizing-clique-helly-graphs}) A graph \(G\) is Helly if and only if for any triangle \(T\) in \(G\), the extended triangle \(\hat{T}\) is a cone.
\end{theorem}

We can now pose a theorem that gives a sufficient condition for convergence.

\begin{theorem}
\label{helly-convergent}
(\cite{MR0329947}) If a graph \(G\) is Helly, then \(G\) is convergent.
\end{theorem}

We now mention some theorems that imply divergence of a graph.

\begin{theorem}
\label{octa}
(\cite{zbMATH03641500}) The graph \(O_{m}\) is divergent if and only if \(m\geq3\).
\end{theorem}

\begin{theorem}
\label{big-cycles}
(\cite{MR2489268}, Theorem 5.4) The complement of a cycle \(C_{n}\) is divergent if and only if \(n\geq8\).
\end{theorem}

\begin{theorem}
\label{connected-sum}
(\cite{MR2489268}, Theorem 4.6) If \(G,H\) are graphs that have a coaffination, and \(H\) is connected, then \(G+H\) is divergent.
\end{theorem}

\begin{theorem}
\label{three-summands}
(\cite{MR2489268}, Theorem 3.6) If \(A\), \(B\), \(C\) are graphs that have a coaffination, then \(A+B+C\) is divergent.
\end{theorem}

\section{The complements of regular graphs of low degree}
\label{sec:orgdc4ea56}

\subsection{The case \(k=1\)}
\label{sec:org0042a4f}

\begin{theorem}
The complement of a \(1\)-regular graph with \(n\) vertices is divergent if and only if \(n\geq 6\).
\end{theorem}

\begin{proof}
A \(1\)-regular graph is the disjoint union of \(m\) copies of \(K_{2}\). It follows that the complements of the \(1\)-regular graphs are precisely the oc\-ta\-he\-drons \(O_{m}\). From Theorem \ref{octa}, we obtain that if the order of the graph is greater than six, then the graph is divergent. Now, the graph \(O_{1}\) consists of two isolated vertices and \(O_{2}\) is a \(4\)-cycle, and both such graphs are Helly and convergent.
\end{proof}

\subsection{The case \(k=2\)}
\label{sec:orgdc2abb7}

\begin{theorem}
\label{caseke2}
The complement of every \(2\)-regular graph with \(n\) vertices is divergent if \(n\geq 9\).
\end{theorem}

\begin{proof}
A \(2\)-regular graph is the disjoint union of several cycles. If there is only one cycle in the union (that is, if the graph is connected), then the result follows from Theorem \ref{big-cycles}. If there are exactly two cycles, then one of them has at least five vertices, and so its complement is connected. We can then apply Theorem \ref{connected-sum} in this case. Now, if there are three or more cycles, we can apply Theorem \ref{three-summands}.
\end{proof}

Contrary to the case \(k=1\), Theorem \ref{caseke2} is not an equivalence, since there are \(2\)-regular graphs with not more than 8 vertices and with divergent complement. Theorem \ref{big-cycles} deals with the case of cycles, and in that case, we get that \(G=C_{8}\) is the only \(2\)-regular connected graph with divergent complement and \(|G|\leq 8\). The only other \(2\)-regular graphs with at most 8 vertices are \(C_{3}\cup C_{i}\) for \(i=3,4,5\) and \(C_{4}\cup C_{4}\). The complement of \(C_{3}\cup C_{3}\) is \(K_{3,3}\) which has no triangles, and thus is Helly by Theorem \ref{helly-extended}, therefore convergent by Theorem \ref{helly-convergent}. The complements of both graphs \(C_{3}\cup C_{4}\) and \(C_{4}\cup C_{4}\) can also be shown to be Helly by applying Theorem \ref{helly-extended}. Now the complement of \(C_{3}\cup C_{5}\) is divergent, because it is isomorphic to \(\overline{C_{3}}+\overline{C_{5}}\) and then Theorem \ref{connected-sum} applies. We have thus proven that the complement \(\overline{G}\) of a \(2\)-regular graph \(G\) is convergent if and only if \(\overline{G}\) is Helly.

\section{The case \(k\geq 3\)}
\label{sec:orgfed7b54}

Given a graph \(G\), denote by \(t(G)\) the number of triangles in \(G\).

\begin{lemma}
\label{sum-triangles-cotriangles}
(\cite{1962-lorden-blue-empty-chromatic-graphs}, Lemma 1)  For a \(k\)-regular graph \(G\) with \(n\) vertices, we have:
   \begin{equation}
\label{eq-lorden}
t(G)+t(\overline{G})=\binom{n}{3}-\frac{1}{2}nk(n-k-1). \qed
   \end{equation}
\end{lemma}

If \(T\) is a cotriangle in a graph \(G\) and \(x\in G\), we will say that \emph{\(x\) is adjacent to \(T\)} if \(x\) is a neighbor of at least two vertices of \(T\).

\begin{lemma}
\label{atleastk}
Let \(G\) be a \(k\)-regular graph such that \(\overline{G}\) is Helly, and let \(T\) be a cotriangle in \(G\). Then there are at least \(k\) vertices in \(G\) that are adjacent to \(T\).
\end{lemma}

\begin{proof}
We have that \(T\) is a triangle in \(\overline{G}\) and so \(\hat{T}\), its extended triangle in \(\overline{G}\), is a cone, by Theorem \ref{helly-extended}. Since \(\overline{G}\) is \((n-k-1)\)-regular, considering the apex of the cone we obtain that \(|\hat{T}|\leq n-k\). This means that there must be at least \(k\) vertices that are neighbors to one or none of the vertices of \(T\) in \(\overline{G}\). Such vertices are adjacent to \(T\).
\end{proof}

We will use 2-switches, as defined in (\cite{MR3445306}, page 20). This means the process of substituting the edges \(\{a,b\},\{u,v\}\) in \(G\) with the edges \(\{a,u\}\), \(\{b,v\}\). Of course, this requires that \(a,u\) are not already adjacent in \(G\), and also that \(b,v\) are not adjacent in \(G\). Applying a 2-switch to a \(k\)-regular graph produces another \(k\)-regular graph.

\begin{lemma}
\label{many-cotriangles}
Let \(G\) be a \(k\)-regular with \(n\) vertices, where \(n\geq4k\), and \(x\in G\). Then \(x\) is adjacent to at most \(\binom{k}{2}(n-2k)+\binom{k}{3}\) cotriangles. If this bound is met, then the connected component of \(x\) is \(K_{k,k}\).
\end{lemma}

\begin{proof}
Let \(x\in G\), as in our hypothesis. We prove first that if there is an edge \(e=\{a,b\}\) in \(N=N_{G}(x)\), we can perform a 2-switch on \(G\), obtaining a graph \(G'\), where \(a\) and \(b\) are no longer adjacent, no new edges appear between vertices of \(N\), and \(x\) is adjacent to at least the same amount of cotriangles in \(G'\) as it was in \(G\).

The edges in \(G\) that cannot be switched with \(\{a,b\}\) are of the form \(\{r,s\}\), and where both \(r,s\) are neighbors of either \(a\) or \(b\); or one of \(\{r,s\}\) is a neighbor of both \(a,b\).

Consider then the edge \(e=\{a,b\}\) in \(N_{G}(x)\). Let \(C\) be the set of common neighbors of the vertices \(a,b\) different from \(x\). Suppose that \(|C|=t\). Let \(Z\) be the set of vertices that are at distance at least 2 from \(x\), and let \(Z_{1}\subseteq Z\) be the vertices in \(Z\) that are not neighbors of either \(a,b\). Then in \(Z_{1}\) there are at least \(4k-[t+2(k-t-2)+k+1]=k+t+3\) vertices.  We claim that there is an edge from a vertex in \(Z_{1}\) to a vertex that is not in  \(C\cup N\cup\{x\}\). If not, then all edges incident with \(Z_{1}\) have one end in \(Z_{1}\) and the other in \(C\cup N\). In this case, there are at least \((k+t+3)k\) edges incident with \(Z_{1}\) and at most \(t(k-2)+(k-2)(k-1)\) possible edges from \(C\cup N\) to \(Z_{1}\). If the former number was actually less than the latter, we would have that \(6k+2t\leq 2\), which is impossible. Therefore, there is an edge joining a vertex \(r\in Z_{1}\) to a vertex \(s\not\in N\) that is not a neighbor of at least one among \(a, b\). Without loss, assume that \(s\) is not a neighbor of \(b\).

We perform a 2-switch to \(G\) removing the edges \(\{a,b\}\), \(\{r,s\}\) and adding the edges \(\{a,r\},\{b,s\}\), obtaining the graph \(G'\). Since neither of \(r,s\) is in \(N\), the graph \(G'\) has one less edge in \(N\) than \(G\). We claim that \(x\) is adjacent to no less cotriangles in \(G'\) than it was in \(G\). Suppose that \(T=\{u,v,w\}\) was a cotriangle in \(G\) but is not a cotriangle in \(G'\). The only way that can happen is if exactly one of the new edges involves two vertices from \(T\). Suppose \(u=a, v=r\). Then \(w\) could have been any of the other elements in \(N\) different from \(a\) and \(b\) and not a neighbor of \(r\). That is, in the process of going from \(G\) to \(G'\), the vertex \(x\) loses at most \(k-2\) adjacent cotriangles. On the other hand, any set of the form \(\{a,b,z\}\) with \(z\in Z_{1}-\{r,s\}\) was not a cotriangle in \(G\) but is a cotriangle in \(G'\). This means that the vertex \(x\) gains at least \(k+t+1\) new adjacent cotriangles in the process of going from \(G\) to \(G'\). 

Because of the argument of the previous paragraph, we may now assume that the \(k\)-regular graph \(G\) is such that there are no edges among the vertices in \(N=N_{G}(x)\). Then \(x\) is adjacent to exactly \(\binom{k}{3}\) cotriangles where the three vertices of the cotriangle are contained in \(N\). We will estimate the amount of cotriangles adjacent to \(x\) where exactly two vertices of the cotriangle are elements of \(N\). For each pair \(ab\) of vertices of \(N\), denote by \(n_{ab}\) the number of common neighbors of \(a,b\) different from \(x\). Then there are \(2k-2-n_{ab}\) vertices different from \(x\) that are neighbors of either \(a\) or \(b\). Hence there are \(n-(2k-2-n_{ab})-k-1=n-3k+1+n_{ab}\) vertices that can be used to extend the set \(\{a,b\}\) to a cotriangle adjacent to \(x\) of the required form. Summing over all pairs of vertices of \(N\) we get \(S=\binom{k}{2}(n-3k+1)+\sum n_{ab}\). We have that \(\binom{k}{2}(n-2k)-S=\binom{k}{2}(k-1)-\sum n_{ab}\). Since \(0\leq n_{ab}\leq k-1\) for each pair \(a,b\), this expression is non-negative, proving our first claim. It is equal to zero if and only if \(n_{ab}=k-1\) for each pair \(a,b\) of vertices of \(N\). If this were the case, then the open neighborhood of each of the vertices in \(N\) consists of the same vertices, therefore the connected component that contains \(x\) is isomorphic to \(K_{k,k}\).
\end{proof}

\section{Proof of the main theorem}
\label{sec:org03be13a}

\begin{theorem}
\label{main-theorem}
If \(n>3k+\sqrt{2k^{2}-k}\), then the complement of any \(k\)-regular graph with \(n\) vertices is not Helly.
\end{theorem}

\begin{proof}
Let \(k>1\) and \(n\) as in the hypothesis of the Theorem. Then \(n>4k\). Let \(G\) be a \(k\)-regular graph with \(n\) vertices such that \(\overline{G}\) is Helly. Let \(B\) be the bipartite graph where one part is the set of vertices of \(G\) and the other part is the set of cotriangles of \(G\). In \(B\), a vertex corresponding to a vertex in \(G\) is a neighbor to the vertex corresponding to a cotriangle whenever the vertex is adjacent to the cotriangle. Denote with \(e(B)\) the amount of edges of \(B\). We will estimate \(e(B)\) from below and from above.

In the graph \(G\), each vertex is in at most \(\binom{k}{2}\) triangles. Therefore, there are at most \(\frac{n\binom{k}{2}}{3}=\frac{1}{6}nk(k-1)\) triangles. Since we have that \(t(G)+t(\overline{G})=\binom{n}{3}-\frac{1}{2}nk(n-1-k)\) by Lemma \ref{sum-triangles-cotriangles} , there are at least
\begin{equation}
\label{cnk}
C_{n,k}=\binom{n}{3}-\frac{1}{2}nk(n-1-k)-\frac{1}{6}nk(k-1)
\end{equation}
cotriangles in \(G\). By Lemma \ref{atleastk}, since each cotriangle is adjacent to at least \(k\) vertices, we have that \(kC_{n,k}\leq e(B)\).

On the other hand, by Lemma \ref{many-cotriangles}, the maximum amount of cotriangles adjacent to any fixed vertex in \(G\) is \(T_{n,k}=\binom{k}{2}(n-2k)+\binom{k}{3}\). Therefore, \(e(B)\leq nT_{n,k}\).

It follows that \(kC_{n,k}\leq e(B)\leq nT_{n,k}\). Now, the difference \(nT_{n,k}-kC_{n,k}\) is equal to \(-\frac{kn}{6}\,a(n,k)\), where  \(a(n,k)=n^{2}-6kn+(7k^{2}+k)\). Hence, we must have \(a(n,k)\leq 0\). Fixing \(k\), and considering \(a(n,k)\) as a quadratic polynomial on \(n\), its roots are \(3k\pm\sqrt{2k^{2}-k}\). From this, the theorem follows.
\end{proof}

\section{The cubic case}
\label{sec:org2e9e85a}

Theorem \ref{main-theorem} implies that the complement of a cubic graph is not Helly if the graph has at least 14 vertices. This bound is tight, since the graph \(K_{3,3}\cup K_{3,3}\) has 12 vertices and its complement is Helly. However, unlike the cases \(k=1\) and \(k=2\), in this case there are graphs that are not Helly but are convergent under the clique graph operator. As examples of that behavior we can give the graphs in Figure \ref{comphelly}, which were found by a computer search. We conjecture that there are cubic graphs \(G\) with an arbitrarily large number of vertices, such that \(\overline{G}\) is Helly.

\begin{figure}\centering
\subfigure[14 vertices]{\includegraphics[width=.3\textwidth]{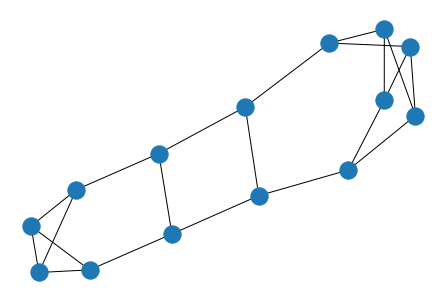}}\hspace*{\fill}
\subfigure[16 vertices]{\includegraphics[width=.3\textwidth]{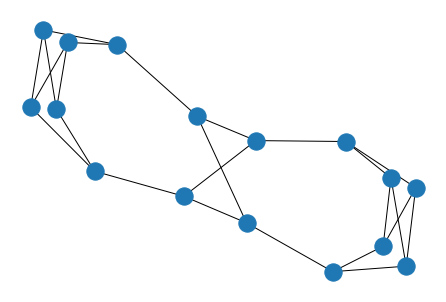}}\hspace*{\fill}
\subfigure[18 vertices]{\includegraphics[width=.3\textwidth]{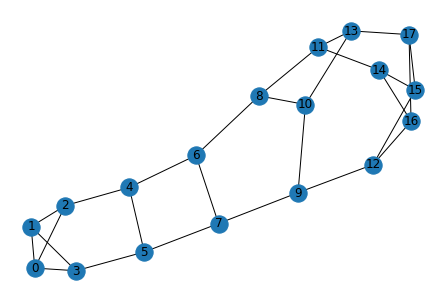}}
\caption{\label{comphelly}Cubic graphs with convergent (non Helly) complements}
\end{figure}

\section{Acknowledgments}
\label{sec:org0210cf8}

Partially supported by CONACYT, grant A1-S-45528.

\bibliographystyle{plain}
\bibliography{complements}
\end{document}